\documentclass[a4paper, reqno, 10pt]{amsart}

\usepackage[usenames,dvipsnames]{color}

\usepackage{amsthm,amsfonts,amssymb,amsmath,amsxtra,amsrefs}
\usepackage[all]{xy}
\SelectTips{cm}{}
\usepackage{xr-hyper}
\usepackage[colorlinks=
   citecolor=Black,
   linkcolor=Red,
   urlcolor=Blue]{hyperref}
\usepackage{verbatim}

\usepackage[margin=1.25in]{geometry}

\usepackage{tikz}

\usepackage{mathrsfs}

\RequirePackage{xspace}
\RequirePackage{etoolbox}
\RequirePackage{varwidth}
\RequirePackage{enumitem}
\RequirePackage{tensor}
\RequirePackage{mathtools}
\RequirePackage{longtable}
\RequirePackage{multirow}

\newcommand{\BA}{\ensuremath{\mathbb {A}}\xspace}

\newcommand{\BC}{\ensuremath{\mathbb {C}}\xspace}

\newcommand{\BK}{\ensuremath{\mathbb {K}}\xspace}

\newcommand{\BP}{\ensuremath{\mathbb {P}}\xspace}
\newcommand{\BQ}{\ensuremath{\mathbb {Q}}\xspace}
\newcommand{\BR}{\ensuremath{\mathbb {R}}\xspace}

\newcommand{\BZ}{\ensuremath{\mathbb {Z}}\xspace}

\newcommand{\CL}{\ensuremath{\mathcal {L}}\xspace}

\newcommand{\CO}{\ensuremath{\mathcal {O}}\xspace}

\renewcommand{\div}{{\mathrm{div}}}

\DeclareMathOperator{\Gal}{Gal}

\newcommand{\lra}{\longrightarrow}

\renewcommand\a\alpha
\renewcommand\b\beta
\newcommand\g\gamma
\renewcommand\d\delta
\newcommand\D\Delta

\newcommand{\Berk}{\mathsf{P}^1}

\renewcommand\div{\textup{div}}

\newtheorem{theorem}[equation]{Theorem}
\newtheorem{prop}[equation]{Proposition}

\newtheorem{lemma}[equation]{Lemma}

\newtheorem{cor}[equation]{Corollary}
\newtheorem{thm}[equation]{Theorem}

\theoremstyle{definition}

\newtheorem{remark}[equation]{Remark}

\numberwithin{equation}{section}

\setitemize[0]{leftmargin=*,itemsep=\the\smallskipamount}
\setenumerate[0]{leftmargin=*,itemsep=\the\smallskipamount}

\renewcommand{\to}{%
   \ifbool{@display}{\longrightarrow}{\rightarrow}%
   }
\let\shortmapsto\mapsto
\renewcommand{\mapsto}{%
   \ifbool{@display}{\longmapsto}{\shortmapsto}%
   }
\newlength{\olen}
\newlength{\ulen}
\newlength{\xlen}
\newcommand{\xra}[2][]{%
   \ifbool{@display}%
      {\settowidth{\olen}{$\overset{#2}{\longrightarrow}$}%
       \settowidth{\ulen}{$\underset{#1}{\longrightarrow}$}%
       \settowidth{\xlen}{$\xrightarrow[#1]{#2}$}%
       \ifdimgreater{\olen}{\xlen}%
          {\underset{#1}{\overset{#2}{\longrightarrow}}}%
          {\ifdimgreater{\ulen}{\xlen}%
             {\underset{#1}{\overset{#2}{\longrightarrow}}}
             {\xrightarrow[#1]{#2}}}}%
      {\xrightarrow[#1]{#2}}
   }
\makeatother
\newcommand{\xyra}[2][]{%
   \settowidth{\xlen}{$\xrightarrow[#1]{#2}$}%
   \ifbool{@display}%
      {\settowidth{\olen}{$\overset{#2}{\longrightarrow}$}%
       \settowidth{\ulen}{$\underset{#1}{\longrightarrow}$}%
       \ifdimgreater{\olen}{\xlen}%
          {\mathrel{\xymatrix@M=.12ex@C=3.2ex{\ar[r]^-{#2}_-{#1} &}}}%
          {\ifdimgreater{\ulen}{\xlen}%
             {\mathrel{\xymatrix@M=.12ex@C=3.2ex{\ar[r]^-{#2}_-{#1} &}}}
             {\mathrel{\xymatrix@M=.12ex@C=\the\xlen{\ar[r]^-{#2}_-{#1} &}}}}}%
      {\mathrel{\xymatrix@M=.12ex@C=\the\xlen{\ar[r]^-{#2}_-{#1} &}}}%
   }
\makeatletter
\newcommand{\xla}[2][]{%
   \ifbool{@display}%
      {\settowidth{\olen}{$\overset{#2}{\longleftarrow}$}%
       \settowidth{\ulen}{$\underset{#1}{\longleftarrow}$}%
       \settowidth{\xlen}{$\xleftarrow[#1]{#2}$}%
       \ifdimgreater{\olen}{\xlen}%
          {\underset{#1}{\overset{#2}{\longleftarrow}}}%
          {\ifdimgreater{\ulen}{\xlen}%
             {\underset{#1}{\overset{#2}{\longleftarrow}}}
             {\xleftarrow[#1]{#2}}}}%
      {\xleftarrow[#1]{#2}}
   }
\newcommand{\isoarrow}{%
   \ifbool{@display}{\overset{\sim}{\longrightarrow}}{\xrightarrow\sim}%
   }
\renewcommand{\lra}{%
   \ifbool{@display}{\longleftrightarrow}{\leftrightarrow}%
   }

\begin{document}

\title{The Arakelov-Zhang pairing and Julia sets} 

\author[A. Bridy]{Andrew Bridy}
\address{Andrew Bridy\\Departments of Political Science and Computer Science\\ Yale University\\
New Haven, CT, 06511, USA}
\email{andrew.bridy@yale.edu}

\author[M. Larson]{Matt Larson}
\address{Matt Larson\\Department of Mathematics\\ Yale University\\
New Haven, CT, 06511, USA}
\email{matthew.larson@yale.edu}

\begin{abstract}
{The Arakelov-Zhang pairing $\langle\psi,\phi\rangle$ is a measure of the ``dynamical distance" between two rational maps $\psi$ and $\phi$ defined over a number field $K$. It is defined in terms of local integrals on Berkovich space at each completion of $K$. We obtain a simple expression for the important case of the pairing with a power map, written in terms of integrals over Julia sets. Under certain disjointness conditions on Julia sets, our expression simplifies to a single canonical height term; in general, this term is a lower bound. As applications of our method, we give bounds on the difference between the canonical height $h_\phi$ and the standard Weil height $h$, and we prove a rigidity statement about polynomials that satisfy a strong form of good reduction.}
\end{abstract}
\date{\today}

\maketitle

\section{Introduction}

Let $K$ be a number field with fixed algebraic closure $\overline{K}$. For $z\in\BP^1(\overline{K})$, we use $h(z)$ to denote the standard logarithmic Weil height of $z$ and $h_\phi(z)$ to denote the Call-Silverman canonical height of $z$ with respect to $\phi\in K(x)$. We recall background on these heights in Section~\ref{Background}.

Let $\psi,\phi:\BP^1\to\BP^1$ be rational maps defined over $K$ (equivalently, rational functions in $K(x)$) each of degree $\geq 2$. In~\cite{PST} (see also~\cite{Zhang}), Petsche, Szpiro, and Tucker introduced the Arakelov-Zhang pairing $\langle \psi,\phi\rangle$, a symmetric, non-negative, real-valued pairing on the space of rational maps. In Section~\ref{Background}, we present their definition of the pairing as a sum of integrals over Berkovich space at each place of $K$. Using~\cite[Theorem 11]{PST} and standard results on equidistribution of preimages, we also give a more intuitive equivalent definition as the limiting average of $h_\psi$ evaluated at the preimages under $\phi$ of any non-exceptional point $\beta\in\BP^1(\overline{K})$ (here ``exceptional" means that $\beta$ has finite backward orbit under $\psi$):
\begin{equation}\label{intuitive}
\langle \psi, \phi \rangle = \lim_{n \to \infty}\frac{1}{(\deg\phi)^n} \sum_{\phi^n(x)=\beta}h_\psi(x).
\end{equation}
The pairing can be understood as a ``dynamical distance" between $\psi$ and $\phi$. For example, by~\cite[Theorem 3]{PST}, $\langle\psi,\phi\rangle$ vanishes precisely when the canonical height functions $h_\psi$ and $h_\phi$ agree; this in turn holds if and only if the sets of preperiodic points of $\psi$ and $\phi$ coincide. Thus the specific pairing $\langle x^2,\phi\rangle$ may be interpreted as a measure of the dynamical complexity of $\phi$, as the height function $h_{x^2}$ equals the standard height $h$. As above, we have\begin{equation}\label{intuitive2}
\langle x^2, \phi \rangle = \lim_{n \to \infty}\frac{1}{(\deg\phi)^n} \sum_{\phi^n(x)=\beta}h(x).
\end{equation}
for non-exceptional $\beta$. 
 We note that $\langle x^d, \phi \rangle = \langle x^2, \phi \rangle$ for any $d \in \BZ \setminus \{-1, 0, 1\}$.

In this paper, we study the relationship between the Arakelov-Zhang pairing and Julia sets, both in the classical and the non-archimedean setting. We produce a formula for the pairing $\langle x^2,\phi\rangle$, which can be computed exactly under certain disjointness conditions on Julia sets.

Let $M_K$ be the set of places of $K$. For $\nu \in M_K$, let $r_{\nu} = [K_{\nu} : \BQ_{\nu}]/[K : \BQ]$, and let $\mu_{\phi,\nu}$ be the canonical $\phi$-invariant probability measure on the Berkovich projective line $\Berk$ over $\BC_\nu$ (see Section \ref{Background} for definitions). Our main theorem is as follows.
\begin{theorem}\label{main}
Let $\phi\in K(x)$. Then
$$\langle x^2, \phi \rangle = h_{\phi}(0) - \sum_{\nu \in M_K} r_{\nu} \int_{\vert \alpha \vert_{\nu} < 1} \log \vert \alpha \vert_{\nu} d \mu_{\phi,\nu}.$$
\end{theorem} 

Theorem~\ref{main} can be viewed as a companion result to Proposition 16 in \cite{PST}, except with the role of $\infty$ and $0$ reversed. The integral is over the Berkovich open unit disk, but as $\mu_{\phi,\nu}$ is supported on the Julia set of $\phi$ in $\Berk$, we may interpret it as an integral over the Julia set. The integral evaluates to $0$ if $\phi$ has good reduction at the place $\nu$, so the sum over all $\nu\in M_K$ is actually a finite sum. 

We give two proofs of Theorem \ref{main}. The first proof requires extensive local analytic machinery as in \cite{PST}. The second proof is more elementary; it relies on our formula for the Arakelov-Zhang pairing as the average height of preimages of a non-exceptional point, and uses the fact that these preimages equidistribute in Berkovich space with respect to the canonical measure. By appealing to the equidistribution theorem, we are able to make an argument that avoids heavy use of the Berkovich space machinery. However, the second proof is only valid for monic polynomials whose $\nu$-adic Julia set does not contain $0$ at any place $\nu$. 

We give some consequences of Theorem~\ref{main} which are easier to state. Corollary~\ref{disjointcor} is a fundamental inequality between the pairing $\langle x^2,\phi\rangle$ and the canonical height $h_\phi(0)$, which are equal under a disjointness condition on Julia sets. See Section~\ref{Juliasection} for details on when this condition is satisfied.

\begin{cor}\label{disjointcor}
For any $\phi\in K(x)$,
$$\langle x^2, \phi \rangle \ge h_{\phi}(0),$$ with equality if and only if the Julia set of $\phi:\Berk\to\Berk$ is disjoint from the Berkovich open unit disk in $\Berk$ at every completion of $K$. 
\end{cor}

Equality can occur in Corollary~\ref{disjointcor}. For example, Proposition~\ref{quad} implies that if $\phi(x) = x^2 + c$ with $c \in \mathbb{Z}$ and $\vert c \vert \ge 4$, then $\langle x^2, \phi \rangle = h_{\phi}(0)$.

Corollary~\ref{integercoeff} is the purely archimedean version of our main theorem, which becomes simpler (and requires no reference to Berkovich space) in the case that $\phi$ is a monic polynomial with integer coefficients. Here $\mu_\phi$ is the invariant measure on $\BP^1(\BC)$, the integral is over the complex unit disk, and the canonical height $h_\phi$ may be interpreted over $\BQ$ (or any number field).

\begin{cor} \label{integercoeff}
Let $\phi\in{\BZ}[x]$ be monic. Then
$$\langle x^2, \phi \rangle = h_{\phi}(0) - \int_{\vert z \vert < 1} \log \vert z \vert d \mu_{\phi}.$$
\end{cor}

\begin{remark}\label{algebraicintegers}
A similar statement to Corollary~\ref{integercoeff} holds if $\phi$ has algebraic integer coefficients, but the integral must be replaced with a sum of integrals corresponding to every possible embedding into $\BC$, and the statement is essentially that of Theorem~\ref{main} (though again with no reference to Berkovich space). See Section~\ref{Juliasection}. The restriction to algebraic integers is interesting from the dynamical point of view because of its connection to families of post-critically finite mappings. For example, in the family of complex quadratic polynomials $\phi_c(z)=z^2+c$, the mappings $\phi_c$ for which the critical point $0$ is preperiodic arise from certain algebraic integer parameters $c$ (the roots of the famous Gleason and Misiurewicz polynomials, see, e.g.,~\cite{Buff}).
\end{remark}

Our method also allows us to prove a rigidity statement about polynomials with certain reduction conditions by combining our work with a result of Kawaguchi-Silverman on maps with equal canonical height functions~\cite{Kawaguchi}.

\begin{thm}\label{rigidity}
Let $\phi(x) = x^d + a_{d-1} x^{d-1} + \dotsb + a_0\in \mathcal{O}_K[x]$, and assume that $0$ is preperiodic under $\phi$. Further suppose that the Julia set of $\phi$ at every archimedean place $\nu$ does not intersect the $\nu$-adic Berkovich open unit disk. Then $\phi(x)=x^d$.
\end{thm}

Petsche, Szipro, and Tucker also show that the pairing $\langle x^2, \phi \rangle$ can be used to give an upper bound on the difference between the canonical height of $\phi$ and the standard height. More precisely, they show the following theorem. 
\begin{theorem}{\cite{PST}*{Theorem 15}}\label{diffbound}
Let $\phi$ be a rational function of degree at least $2$ defined over a number field $K$. Then for any $z \in \BP^1(\overline{K})$,
$$h_{\phi}(z) - h(z) \le \langle x^2, \phi \rangle + h_{\phi}(\infty) + \log 2.$$
\end{theorem}

The explicit nature of Theorem \ref{main} allows us to compute the pairing with some rational functions where the canonical measure is known explicitly (such as Chebyshev polynomials). For these examples, we then apply Theorem \ref{diffbound} to bound the difference between the Weil height and the canonical height. 

The paper is organized as follows. In Section 2, we recall some relevant background and define the Arakelov-Zhang pairing in terms of local integrals. We also prove Theorem \ref{main} and its corollaries. In Section 3, we recall some facts about Julia sets and prove Theorem \ref{rigidity}. In Section 4, we compute some explicit examples.

\subsection*{Acknowledgements}
We would like to thank Holly Krieger and Tom Tucker for numerous helpful conversations related to the topics in this paper. We also thank the referee for helpful comments and corrections, and for suggesting an improved argument for Lemma~\ref{rat}.

\section{Background}\label{Background}

We sketch the background needed to properly define the Arakelov-Zhang pairing, starting with a brief overview of height functions. See~\cite{LangDiophantine} for background and basic properties of height functions, and see~\cite{CallSilverman} for background on the canonical height function.

Let $K$ be a number field. The logarithmic height of $x\in K$ is defined by
$$
h(x) = \sum_{\nu\in M_K}r_\nu\log \max \{|x|_\nu, 1\},
$$
where $r_{\nu} = [K_{\nu} : \BQ_{\nu}]/[K : \BQ]$ as in the introduction. This definition immediately extends in a compatible way to any finite extension $K'$ of $K$ by the local-global degree formula, and so $h$ is a function $h:\overline{K}\to \BR$. We extend $h$ to a function $h:\BP^1(\overline{K})\to\BR$ by setting $h(\infty)=0$. See~\cite{BridyTucker} or~\cite{GNT} for an alternate but equivalent way of defining the height function.

Fix a rational function $\phi\in K(x)$ with $d=\deg\phi\geq 2$. We use $\phi^n$ to mean the $n$-fold composition of $\phi$ with itself. The Call-Silverman canonical height relative to $\phi$ is defined by
$$
h_\phi(x)=\lim_{n\to\infty} \frac{h(\phi^n(x))}{d^n}
$$
for all $x\in\BP^1(\overline{K})$. In~\cite{CallSilverman} it is shown that this limit exists, and its basic properties are established. Importantly, for all $x\in\BP^1(\overline{K})$,
\begin{align*}
h_\phi(\phi(x)) & =dh_\phi(x),\text{ and}\\
|h(x)-h_\phi(x)| & < C_\phi
\end{align*}
for an absolute constant $C_\phi$ (in fact, the canonical height is uniquely characterized by these two properties). The other property of the canonical height that we will use is that $h_\phi(x)=0$ if and only if $x$ is preperiodic for $\phi$, i.e., if $\phi^n(x)=\phi^m(x)$ for some $n>m\geq 0$. In the setting of number fields, this fact is a simple consequence of Northcott's theorem~\cite{Silverman}*{Theorem 3.22}. It also holds if $K$ is a function field and $\phi$ is not isotrivial, due to work of Benedetto~\cite{Benedetto} in the polynomial case and Baker~\cite{Bakerheights} for rational functions.

We say that a map $\phi\in K(x)$ has good reduction at a non-archimedean place $\nu$ if the degree of $\phi$ is unchanged after reducing the coefficients to the residue field $k_\nu$. To be precise, we must first write $\phi$ as a map $\BP^1\to\BP^1$ in homogeneous coordinates and choose a normalized form. See~\cite{MortonSilverman} or~\cite[Theorem 2.18]{Silverman} for details. For a polynomial
\[
\phi(x) = a_dx^d + a_{d-1}x^{d-1}+\dots + a_1x + a_0,
\]
good reduction has the simple interpretation that $\nu(a_d)= 0$ and $\nu(a_i)\geq 0$ for $0\leq i\leq d-1$. There are evidently only finitely many non-archimedean places $\nu\in M_K$ for which $\phi$ has bad reduction (i.e., does not have good reduction).

We recall some notation and terminology from~\cite{PST}. Let $\BK$ be either the complex numbers $\BC$ or the field $\BC_v$ which is the completion of the algebraic closure of $K_v$, and let $|\cdot|$ be the standard absolute value on $\BK$. Let $\Berk$ denote the Berkovich projective line over $\BK$; for $\BK=\BC$, $\Berk$ is simply $\BP^1(\BC)$. See~\cite{Baker} for background on Berkovich space in the context of dynamics.

Let $\phi:\BP^1\to\BP^1$ be a morphism of degree $d$ defined over $\BK$. A polarization $\epsilon$ of $\phi$ is an isomorphism $\epsilon:\CO(d)\stackrel{\sim}{\to}\phi^*\CO(1)$, where by $\CO(d)$ we mean $\CO_{\BP^1}(d)$. More concretely, a choice of a polarization is equivalent to a choice of a homogeneous lift of $\phi$ to a polynomial endomorphism $\Phi:\BA^2\to\BA^2$ (see~\cite[Section 2.1]{PST}). 

Recall that a metric on a line bundle $\CL$ is a non-negative, real-valued function on $\CL$ such that the restriction to each fiber $\CL_x$ is a norm on $\CL_x$ as a $\BK$-vector space. The standard metric $||\cdot||_{st}$ on $\CO(1)$ is characterized by the identity
$$
||s(x)||_{st}=\frac{|s(x_0,x_1)|}{\max\{|x_0|,|x_1|\}}
$$
on the fiber of $\CO(1)$ above $x=[x_0:x_1]$, for each section $s\in\Gamma(\BP^1,\CO(1))$ written as $s(x)=s(x_0,x_1)\in \BK[x_0,x_1]$. The canonical metric $||\cdot||_{\phi,\epsilon}$ is the limit as $k\to\infty$ of the sequence of metrics characterized by
\begin{align*}
||\cdot||_{\phi,\epsilon,0} &=||\cdot||_{st}\\
||\cdot||_{\phi,\epsilon,k+1}^{\otimes d}  &= \epsilon^*\phi^*||\cdot||_{\phi,\epsilon,k}
\end{align*}
for all $k$. Zhang showed that this limit exists as a bounded, continuous metric on $\CO(1)$~\cite{Zhang}. See~\cite{PST} for details on how the metric $||\cdot||_{\phi,\epsilon}$ depends on the polarization $\epsilon$.

The standard measure $\mu_{st}$ on the Berkovich projective line $\Berk$ over $\BK$ is the Haar measure on the unit circle for $\BK=\BC$, and the Dirac point mass at the Gauss point $\zeta_{0,1}$ for $\BK$ non-archimedean. The canonical invariant probability measure $\mu_\phi$ is defined as the weak limit as $k\to\infty$ of the sequence of measures given by
\begin{align*}
\mu_{\phi,0}&=\mu_{st}\\
\mu_{\phi,k+1}&=\frac{1}{d}\phi^*\mu_{\phi,k}
\end{align*}
for all $k$. The measure satisfies $\phi_*\mu_\phi=\mu_\phi$ and $\phi^*\mu_\phi=d \cdot \mu_\phi$. If $\beta$ is any non-exceptional point, then $\frac{1}{d^n}\sum_{\phi^n(\alpha)=\beta}\delta_\alpha$ converges weakly to $\mu_{\phi}$. 

Over $\BC$, the canonical invariant measure was constructed by Brolin for polynomial mappings~\cite{Brolin} and extended to rational functions by both Ljubich~\cite{Ljubich} and Freire-Lopes-Ma\~{n}\'e~\cite{invariant}. In the non-archimedean setting, the measure was introduced independently by Baker-Rumely~\cite{BakerRumely}, Chambert-Loir~\cite{ChambertLoir}, and Favre-Rivera-Letelier~\cite{FavreRiveraLetelier}. The measure $\mu_\phi$ can also be described as the unique $\phi$-invariant measure of maximal entropy $\log d$ (in particular, this equals the topological entropy of $\phi$). 

\begin{remark}
In the proofs of our main results, the map $\phi$ will be defined over a number field $K$. The metrics on $\CO(1)$ and measures on $\Berk$ that we have defined exist at every place $\nu$ of $K$, i.e., with $\BK=\BC_\nu$, and with $\phi$ considered as a map defined over $\BK$ under the embedding corresponding to $\nu$. When necessary, we will indicate the dependence on $\nu$ with an additional subscript, for example, $\vert\vert s(x)\vert\vert_{st,\nu}$ or $\mu_{\phi,\nu}$.
\end{remark}

The Julia set of $\phi$ can be defined in many equivalent ways. The most classical definition is that the Fatou set of $\phi$ is the locus on which the family of iterates $\{\phi^n\}_{n=1}^\infty$ is equicontinuous, and the Julia set is the complement of the Fatou set. The support of the invariant measure $\mu_\phi$ is precisely the Julia set of $\phi$ -- this was proved in~\cite{invariant} for $\BC$ and by Rivera-Letelier in the non-archimedean case (see~\cite{Baker}*{Theorem 10.56} for a writeup of the proof).

For the moment, let $\psi, \phi: \BP^1 \to \BP^1$ be defined over $\BK$. Fix a polarization $\epsilon$ of $\phi$. Let $s, t \in \Gamma(\BP^1, \CO(1))$ be sections with $\div(s) \not= \div(t)$. The local Arakelov-Zhang pairing is defined by
\begin{equation}\label{localpairing}
\langle \psi, \phi \rangle_{s, t} = \log \vert \vert s(\div(t)) \vert \vert_{\phi, \epsilon} - \int \log \vert \vert s(x) \vert \vert_{\phi, \epsilon} d \mu_{\psi}(x).
\end{equation}
Now let $\psi, \phi$ be defined over the number field $K$. For $\nu\in M_K$, the local pairing of $\psi$ and $\phi$ for $\BK=\BC_\nu$ is denoted $\langle \psi, \phi \rangle_{s, t, \nu}.$
The global Arakelov-Zhang pairing is then defined by
\begin{equation}\label{globalpairing}
\langle \psi, \phi \rangle = \sum_{\nu \in M_K} r_\nu \langle \psi, \phi \rangle_{s, t, \nu} + h_{\psi}(\div(s)) + h_{\phi}(\div(t)).
\end{equation}
As is clear from the notation, the local pairing does not depend on the choice of polarization $\epsilon$, and the global pairing does not depend on the choice of sections $s$ and $t$ (see~\cite{PST}).

\begin{remark}
The Arakelov-Zhang pairing may be understood as a dynamical distance, but it is not a metric on the space of rational maps. However, as observed by Fili, it coincides with the square of a metric of mutual energy defined on a space of adelic measures~\cite{Fili}. These measures have associated canonical height functions, which in the dynamical setting agree with the Call-Silverman canonical height. This connection has been fruitful in studying ``unlikely intersection" problems of some relation to the questions studied in this paper, e.g., the recent work in~\cite{DKY} on uniform Manin-Mumford.
\end{remark}

As mentioned in the introduction, the formula given for $\langle\psi,\phi\rangle$ in Equation~\ref{intuitive} follows from the main results of~\cite{PST} combined with a result on equidistribution of preimages of a non-exceptional point. We prove this as Proposition~\ref{averageheightsofpreimages}. First we recall a theorem of Petsche-Szpiro-Tucker.
\begin{theorem}\cite{PST}*{Theorem 1}\label{heightlimit}
Let $\psi$ and $\phi$ be rational functions defined over a number field $K$. Let $\{x_n\} \in \BP^1(\overline{K})$ be a sequence of distinct points such that $h_{\phi}(x_n) \to 0$. Then $h_{\psi}(x_n) \to \langle \psi, \phi \rangle$. 
\end{theorem}

Recall that an exceptional point $\beta\in\BP^1(\overline{K})$ of a rational map $\phi:\BP^1\to\BP^1$ is a point such the set of all $x\in\BP^1(\overline{K})$ such that $\phi^n(x)=\beta$ for some $n\geq 1$ (the backward orbit of $\beta$) is a finite set. It is not hard to show that, if $\beta$ is exceptional for $\phi$, then up to conjugacy by M\"obius transformations, either $\phi$ is a polynomial and $\beta=\infty$ or $\phi(x)=x^d$ for $d\in\BZ$ and $\beta\in\{0,\infty\}$.

\begin{prop}\label{averageheightsofpreimages}
Let $\psi, \phi :\BP^1\to\BP^1$ be rational maps defined over a number field $K$. Suppose $\beta\in\BP^1(\overline{K})$ is not exceptional for $\phi$, and let $d$ be the degree of $\phi$. Then 
$$
\langle \psi, \phi \rangle = \lim_{n \to \infty}\frac{1}{d^n} \sum_{\phi^n(x)=\beta}h_\psi(x),
$$
where the summation is counted with multiplicity.
\end{prop}

\begin{proof}
 If, for all $n$, there are precisely $d^n$ points $x$ with $\phi^n(x)=\beta$, then the claim follows directly from Theorem~\ref{heightlimit}. In order to prove the Proposition in general, we need to show that the points in the multiset $\{\phi^n(x)=\beta\}$ cannot occur with too large a multiplicity as $n\to\infty$. Fix an embedding $\overline{K}\hookrightarrow\BC$. In particular, we may view $\phi$ as a rational function with complex coefficients under this embedding. Since $\beta$ is a non-exceptional point, the $n$th preimages of $\beta$ equidistribute along the complex Julia set of $\phi$ under the canonical measure $\mu_\phi$.

It follows from~\cite{Ljubich}*{Theorem 4} that $\mu_\phi$ has no point masses (the existence of point masses would violate the ``balanced measure" condition). Thus the number of times any $\alpha\in\BP^1(\BC)$ occurs in the multiset of $d^n$ preimages of $\beta$ is $o(d^n)$ as $n \to \infty$. 

Fix  $\varepsilon>0$. Then there is $N$ such that there are only finitely many $\alpha$ with $h_{\phi}(\alpha) < h_{\phi}(\beta)d^{-N}$ and $\vert h_{\psi}(\alpha) - \langle \psi, \phi \rangle \vert > \varepsilon$, since otherwise we could find a sequence of distinct points $\{x_\ell\}$ such that $h_{\phi}(x_\ell) \to 0$ but $h_{\psi}(x_\ell)$ does not tend to $\langle \psi, \phi \rangle$, which would violate Theorem~\ref{heightlimit}. Fix such an $N$, and call these finitely many points $\alpha_1, \alpha_2, \dotsc, \alpha_k$. For all $n> N$ sufficiently large, these $\alpha_i$ will occur at most $d^n \varepsilon/ \max \{\vert h_{\psi}(\alpha_i) - \langle \psi, \phi \rangle \vert\}$ times as $n$th preimages of $\beta$. For such $n$, we see that 
$$\left \vert \langle \psi, \phi \rangle  - \frac{1}{d^n} \sum_{\phi^n(\alpha) = \beta} h_{\psi}(\alpha) \right \vert < 2 \varepsilon.$$
The claim follows.
\end{proof} 

We now proceed to proofs of the main theorems in the introduction.The first proof  uses the definition of the Arakelov-Zhang pairing given in Equation \ref{globalpairing}. The proof is somewhat similar to the proof of Proposition 16 in \cite{PST}. We give a second, more elementary proof to illustrate the utility of Proposition~\ref{averageheightsofpreimages}. To avoid complications, we assume that $\phi$ is a monic polynomial whose Julia set at any place does not contain $0$. 

\begin{proof}[First proof of Theorem \ref{main}]
Let $\phi:\BP^1\to\BP^1$ be defined over $K$. Fix homogeneous coordinates $x_0, x_1$ on $\BP^1$, and set $x=[x_0:x_1]$. Consider the map $\BP^1 \to \BP^1$ given by $x^2$ (i.e., $[x_0:x_1]\mapsto [x_0^2 : x_1^2]$), and choose the polarization $\epsilon$ that yields its homogeneous lift $\Phi(x_0,x_1)=(x_0^2,x_1^2)$.

The squaring map has good reduction at all non-archimedean places of $K$, so by \cite{PST}*{Proposition 6} we have
\[
\vert\vert s(x)\vert\vert_{x^2,\epsilon,\nu}=\vert\vert s(x)\vert\vert_{st,\nu}
\]
for each non-archimedean place $\nu$. This equality also holds at each archimedean place -- simply observe that the identity in Equation 13 of~\cite{PST}*{Proposition 6} is clearly true for $\nu$ archimedean, and the rest of the argument goes through verbatim.

To compute the local pairing, we choose sections $s,t\in \Gamma(\BP^1(\BC_\nu), \CO(1))$. Let $s(x_0, x_1) = x_0  $ and $t(x_0, x_1) = x_0 + x_1$. We compute $\div(s) = [0 : 1]\text{ and }\div(t) = [1 : -1]$. Identify $\BP^1$ with $\BK \cup \{[1 : 0]\}$ in the usual way, with $a \in \BK$ corresponding to $[a : 1]$. Then for any place $\nu\in M_K$,
$$ \vert \vert s(a) \vert \vert_{x^2,\epsilon,\nu} = \vert \vert s(a) \vert \vert_{st,\epsilon,\nu} = \frac{\vert a \vert_\nu}{\max\{\vert a \vert_\nu, 1\}} = \min\{ \vert a \vert_\nu, 1\}.$$
Therefore,
$$\int \log \vert \vert s(x) \vert \vert_{x^2,\epsilon,\nu} d \mu_{\phi,\nu}(x) = \int_{\vert \alpha \vert_\nu < 1} \log \vert \alpha \vert_\nu d \mu_{\phi, \nu}(\alpha).$$
We compute $\vert\vert s(\div(t))\vert\vert_{x^2,\epsilon}=\vert\vert s(\div(t))\vert\vert_{st,\epsilon}=1$. So for any place $\nu$,
 $$\langle x^2, \phi \rangle_{s, t, \nu} = -\int_{\vert \alpha \vert_{\nu} < 1} \log \vert \alpha \vert_{\nu} d \mu_{\phi, \nu}(\alpha)$$
 by the definition of the local pairing in Equation~\ref{localpairing}.
  
Now observe that $h_{\phi}(\div(s)) = h_{\phi}(0)$ and $h_{x^2}(\div(t))=0$, as $[1: -1]$ is preperiodic under the map $x^2$. Using Equation~\ref{globalpairing}, we compute
\begin{align*}
\langle x^2, \phi \rangle &  = \sum_{\nu \in M_K} r_\nu \langle x^2, \phi \rangle_{s, t, \nu} + h_{x^2}(\div(t)) + h_{\phi}(\div(s))\\
 & = h_\phi(0) - \sum_{\nu \in M_K} r_\nu\int_{\vert \alpha \vert_{\nu} < 1} \log \vert \alpha \vert_{\nu} d \mu_{\phi, \nu}(\alpha),
\end{align*}
as claimed.
\end{proof}

\begin{proof}[Second proof of Theorem \ref{main}]
Assume that $\phi\in K[x]$ is a monic polynomial whose $\nu$-adic Julia set does not contain $0$ at any place $\nu$. Let $S \subseteq M_K$ be a finite set of places of $K$ containing both the archimedean places and the places where $\phi$ has bad reduction. Let $\mathcal{O}_{K,S}$ be the ring of $S$-integers of $K$, i.e.,
\[
\CO_{K,S}=\{z\in K:\nu(z)\geq 0\text{ for }\nu\notin S\}.
\]
Choose  $\beta \in \CO_{K,S}$ such that 
\begin{itemize}
\item $\beta$ is not exceptional for $\phi$, and
\item for all $\nu \in S$, $\beta$ is not contained in the closure of $\{\phi^n(0)\}$ in the topology defined by $\nu$.
\end{itemize}


We do some preliminary computations that will allow us to compute the average height of the preimages of $\beta$. Let $L$ denote the splitting field of $\phi^n(x) - \beta$, and let $\nu \in M_K$. Recall that $r_{\nu} = \frac{[K_{\nu} : \BQ_{\nu}]}{[K:\BQ]}$. Using that $\phi$ is monic, we compute
\begin{equation*} \begin{split}
  \sum_{\omega \in M_L, \omega \mid \nu} \frac{[L_{\omega} : \BQ_{\nu}]}{[L: \BQ]} &\sum_{\phi^n(\alpha) = \beta} \log \max\{ \vert \alpha \vert_{\omega}, 1\} \\
  & =  \sum_{\omega \in M_L, \omega \mid \nu} \frac{[L_{\omega} : \BQ_{\nu}]}{[L: \BQ]}  \left( \sum_{\phi^n(\alpha) = \beta} \log \vert \alpha \vert_{\omega} - \sum_{\phi^n(\alpha) = \beta, \vert \alpha \vert_{\omega} < 1} \log \vert \alpha \vert_{\omega} \right) \\
& = \sum_{\omega \in M_L, \omega \mid \nu} \frac{[L_{\omega} : \BQ_{\nu}]}{[L: \BQ]}  \left( \log \vert \phi^n(0) - \beta \vert_{\omega} - \sum_{\phi^n(\alpha) = \beta, \vert \alpha \vert_{\omega} < 1} \log \vert \alpha \vert_{\omega} \right) \\
& = r_{\nu} \log \vert \phi^n(0) - \beta \vert_{\nu} - \sum_{\omega \in M_L, \omega \mid \nu} \frac{[L_{\omega} : \BQ_{\nu}]}{[L: \BQ]} \sum_{\phi^n(\alpha) = \beta, \vert \alpha \vert_{\omega} < 1} \log \vert \alpha \vert_{\omega},
\end{split} \end{equation*}
where we use that for any $\alpha \in K$, the extension formula implies that $$\prod_{\omega \in M_L, \omega \mid \nu} \vert \alpha \vert_{\omega}^{r_{\omega}} = \vert \alpha \vert_{\nu}^{r_{\nu}}.$$ 

Fix a completion of the algebraic closure of $K_{\nu}$, $\BC_{\nu}$, with an absolute value that extends $\vert \cdot \vert_{\nu}$. There is a distinguished place $\omega$ on $L$ extending $\nu$ such that $\omega$ agrees with the valuation on $\BC_{\nu}$. Since $G = \Gal(L/K)$ acts transitively on the set of valuations above $\nu$, every valuation $\omega'$ on $L$ above $\nu$ is of the form $\vert \cdot \vert_{\omega'} = \vert \sigma( \cdot) \vert_{\omega}$ for some $\sigma \in G$. Let $D_{\omega}$ denote the decomposition group of $\omega$, i.e., the stabilizer of $\omega$ in $G$. Then the set of valuations of $L$ above $\nu$ is isomorphic as a $G$-set to $G/D_{\omega}$. 
Note that $$\vert G/D_{\omega} \vert = \frac{\vert G \vert}{[L_{\omega} : K_{\nu}]} = \frac{[L : K]}{[L_{\omega} : K_{\nu}]}.$$ Thus we may write
\begin{equation*} 
\sum_{\omega \in M_L, \omega \mid \nu} \frac{[L_{\omega} : \BQ_{\nu}]}{[L: \BQ]} \sum_{\phi^n(\alpha) = \beta, \vert \alpha \vert_{\omega} < 1} \log \vert \alpha \vert_{\omega} = \frac{[L_{\omega} : \BQ_{\nu}]}{[L : \BQ]} \sum_{g\in G/D_{\omega}} \sum_{\phi^n(\alpha) = \beta, \vert \alpha \vert_{g \omega} < 1} \log \vert \alpha \vert_{g \omega},
\end{equation*}
where we identify $G/D_\omega$ with a given choice of a left transversal. For each $\alpha$ with $\vert \alpha \vert_{\omega} < 1$ and each $g \in G/D_{\omega}$, the term $\log \vert \alpha \vert_{\omega}$ appears exactly once in the inner sum. Therefore
$$\sum_{\omega \in M_L, \omega \mid \nu} \frac{[L_{\omega} : \BQ_{\nu}]}{[L: \BQ]} \sum_{\phi^n(\alpha) = \beta, \vert \alpha \vert_{\omega} < 1} \log \vert \alpha \vert_{\omega} =
\frac{[K_{\nu} : \BQ_{\nu}]}{[K : \BQ]} \sum_{\phi^n(\alpha) = \beta, \vert \alpha \vert_{ \omega} < 1} \log \vert \alpha \vert_{\omega}.$$
For $\omega \in M_L$ lying above $\nu \not \in S$, we see that if $\phi^n(\alpha) = \beta$, then $\vert \alpha \vert_{\omega} \le 1$. Thus if $\phi^n(\alpha) = \beta$, then 
$$h(\alpha) = \sum_{\nu \in S} \sum_{\omega \in M_L, \omega \mid \nu} r_{\omega} \log \max \{\vert \alpha \vert_{\omega}, 1\}.$$ 
Now we can compute the average height of the preimages of $\beta$ under $\phi^n$ as follows:
\begin{equation*} \begin{split}
 \frac{1}{d^n} \sum_{\phi^n(\alpha) = \beta} h(\alpha) & = \frac{1}{d^n} \sum_{\nu \in S}  \sum_{\omega \in M_L, \omega \mid \nu} \frac{[L_{\omega} : \BQ_{\nu}]}{[L : \BQ]} \sum_{\phi^n(\alpha) = \beta} \log \max \{\vert \alpha \vert_{\omega},1\} \\
& =  \sum_{\nu \in S} \left( \frac{r _{\nu} \log \vert \phi^n(0) - \beta \vert_{\nu}}{d^n} - r_{\nu}  \frac{1}{d^n} \sum_{\phi^n(\alpha) = \beta, \vert \alpha \vert_{\nu} < 1}  \log \vert \alpha \vert_{\nu} \right).
\end{split} \end{equation*}
We claim that
\[
h(\phi^n(0) - \beta)
= \sum_{\nu\in S} \left (r_\nu \log \vert \phi^n(0) - \beta \vert_{\nu} \right) + O(1).
\]
Indeed, for $\nu \not \in S$, $r_{\nu} \log \max \{\vert \phi^n(0) - \beta \vert_{\nu}, 1\} = 0$ as $\phi^n(0)$ and $\beta$ both have non-negative valuation. For the finitely many $\nu \in S$, the choice of $\beta$ guarantees that $\vert \phi^n(0) - \beta \vert_{\nu}$ is bounded below independently of $n$, so $r_\nu \log \vert \phi^n(0) - \beta \vert_{\nu} - r_\nu \log \max \{\vert \phi^n(0) - \beta \vert_{\nu}, 1\} = O(1)$. Therefore,
$$ \frac{1}{d^n} \sum_{\phi^n(\alpha) = \beta} h(\alpha) =  \frac{ h(\phi^n(0) - \beta)}{d^n} - \sum_{\nu \in S} \left( r_{\nu} \frac{1}{d^n} \sum_{\phi^n(\alpha) = \beta, \vert \alpha \vert_{\nu} < 1}  \log \vert \alpha \vert_{\nu} \right) + O(1/d^n).$$

Now we take the limit of the average height of preimages as $n \to \infty$. The left-hand side approaches $\langle x^2,\phi \rangle$ by Proposition~\ref{averageheightsofpreimages}. On the right-hand side, the first term approaches $h_{\phi}(0)$ by basic properties of heights. As $S$ is finite, we may interchange the limit with the sum over $\nu\in S$ in the second term. Thus we arrive at
$$
\langle x^2,\phi\rangle = h_\phi(0) -  \sum_{\nu \in S} r_{\nu}\lim_{n\to\infty}  \frac{1}{d^n} \sum_{\phi^n(\alpha) = \beta, \vert \alpha \vert_{\nu} < 1}  \log \vert \alpha \vert_{\nu}.
$$
By assumption, the Julia set of $\phi$ at every completion $\nu$ of $K$ does not include $0$. Therefore the function $\alpha \mapsto \log \vert \alpha \vert_{\nu}$ is a bounded continuous function on the support of the canonical invariant measure $\mu_{\phi,\nu}$. By the weak convergence of the average of point masses $\frac{1}{d^n}\sum_{\phi^n(\alpha)=\beta}\delta_\alpha$ to the canonical measure,
we have that
$$\lim_{n \to \infty} \frac{1}{d^n} \sum_{\phi^n(\alpha) = \beta, \vert \alpha \vert_{\nu} < 1} \log \vert \alpha \vert_{\nu} = \int_{\vert \alpha \vert_{\nu} < 1} \log \vert \alpha \vert_{\nu} d \mu_{\phi, \nu}$$
and we obtain the formula
$$\langle x^2, \phi \rangle = h_{\phi}(0) - \sum_{\nu \in S} r_{\nu} \int_{\vert \alpha \vert_{\nu} < 1} \log \vert \alpha \vert_{\nu} d \mu_{\phi,\nu}.$$
For $\nu \not \in S$, note that $$\int_{\vert \alpha \vert_{\nu} < 1} \log \vert \alpha \vert_{\nu} d\mu_{\phi, \nu} = 0,$$
as $\mu_{\nu, \phi}$ is supported at the Gauss point $\zeta_{0, 1}$. 
Therefore the integrals in the statement of Theorem~\ref{main} vanish for $\nu\notin S$, so we are done.
\end{proof}

\begin{remark}
We include the second proof to show that, using Proposition~\ref{averageheightsofpreimages}, in some cases it is possible to prove Theorem~\ref{main} without invoking much of the fairly technical local analytic machinery. The main difficulty in extending the argument from polynomials to rational functions is that it is no longer necessarily true that there is a finite set $S\subseteq M_K$ such that the heights of preimages can be computed by only looking at local contributions from places in $S$. In particular, there are contributions from places other than the archimedean places and places of bad reduction.
\end{remark}

\begin{proof}[Proof of Corollary~\ref{disjointcor}]
The Julia set is the support of the canonical measure. Thus if the Julia set is disjoint from the Berkovich open unit disk for a valuation $\nu$, then $$\int_{\vert \alpha \vert_{\nu} < 1} \log \vert \alpha \vert_{\nu} d\mu_{\phi, \nu} = 0,$$ so $\langle x^2, \phi \rangle = h_{\phi}(0)$. 

Suppose the Julia set of $\phi$ at some valuation $\nu$ intersects the Berkovich open unit disk. Then the Julia set intersects the open ball of radius $r$ for some $r < 1$. The measure of this ball is then positive, and we have that $\log \vert \alpha \vert < -\varepsilon < 0$ for some $\varepsilon$ on this ball. Thus $$\int_{\vert \alpha \vert_{\nu} < 1} \log \vert \alpha \vert_{\nu} d\mu_{\phi, \nu} < 0,$$ implying the result. 
\end{proof}

\begin{proof}[Proof of Corollary \ref{integercoeff}]
Let $\nu\in M_K$ be a non-archimedean place. If $\phi$ has good reduction at $\nu$, then the Julia set in Berkovich space is the Gauss point, hence is disjoint from the Berkovich open unit disk. A monic polynomial with integer coefficients has good reduction at every such $\nu$, and we are done.
\end{proof}

\section{Applications to Julia sets}\label{Juliasection}

First, we give two conditions that guarantee that 
$$\int_{\vert \alpha \vert_{\nu} < 1} \log \vert \alpha \vert_{\nu} d \mu_{\phi,\nu} = 0.$$ 
The first condition is good reduction. 
If $\phi$ is a rational function with good reduction at $\nu$, then the Julia set of $\phi$ at $\nu$ is the Gauss point $\zeta_{0, 1}$~\cite{Baker}*{Chapter 10.5}. Since $\zeta_{0, 1}$ is not contained in the Berkovich open unit disk, the integral vanishes if we have good reduction. Note that, as rational functions have good reduction away from finitely many places, this implies that only finitely many of the terms in the sum in Theorem \ref{main} are nonzero. 

Our second condition is the hypothesis of Lemma~\ref{rat} for $\phi$ a polynomial.
\begin{lemma}\label{rat}
Let $\phi(x) = x^n + a_{n-1} x^{n-1} + \dotsb + a_0$ be a polynomial with coefficients in a number field $K$, and let $\nu$ be a non-archimedean valuation which satisfies the condition 
\begin{equation*} \label{condition}
\text{$\nu(a_0) \leq 0$ and $\nu(a_0) < \nu(a_i)$ for every $i$.} 
\end{equation*} Then the Julia set of $\phi$ in the $\nu$-adic Berkovich space $\Berk$ does not intersect the open unit disk.
\end{lemma}

\begin{proof}
If $\nu(a_0)=0$, then $\phi$ has good reduction and we are done. Therefore we may assume that $\nu(a_0)< 0$ by our condition on the coefficients of $\phi$.

Let $D=\{\zeta\in{\bf A}^1_\nu : |\zeta|_\nu<1\}$ be the Berkovich open unit disk, where $|\zeta|_\nu$ is interpreted in the Berkovich sense (the seminorm $[\cdot]_\zeta$ evaluated at the polynomial $x$). In order to show that $J(\phi)$ does not intersect $D$, it suffices to show that $D$ is in the basin of attraction of the attracting fixed point at $\infty$. 

Suppose $|\zeta|_\nu<1$. By our assumption on the coefficients and the ultrametric inequality, we have $|\phi(\zeta)|_\nu=|a_0|_\nu>1$. If any $\xi\in {\bf A}^1_\nu$ satisfies $|\xi|_\nu\geq |a_0|_\nu>1$, then again by the ultrametric we have $|\phi(\xi)|_\nu=|\xi|_\nu^n$. Iterating $\phi$, we find that $|\phi^m(\xi)|_\nu=|\xi|_\nu^{n^m}$, and this expression tends to $\infty$ as $m$ increases. Taking $\xi=\phi(\zeta)$, we see that $\phi(\zeta)$ is in the basin of attraction of $\infty$; thus $\zeta$ is as well.
\end{proof}

Let $\phi\in K[x]$ be a monic polynomial such that, for every non-archimedean place $\nu$, either $\phi$ has good reduction at $\nu$ or $\phi$ satisfies the hypothesis of Lemma~\ref{rat}. 
For example, this criterion applies to every $\phi$ in the quadratic family $x^2 + c$. Lemma~\ref{rat} implies that 
\begin{equation}\label{noBerk}
\langle x^2, \phi \rangle = h_\phi(0) - \sum_{\nu \text{ archimedean}} r_{\nu} \int_{\vert \alpha \vert_{\nu} < 1} \log \vert \alpha \vert_{\nu} d \mu_{\phi, \nu}.
\end{equation}
Equation~\ref{noBerk} is a slight generalization of Theorem~\ref{integercoeff}. Note that the hypothesis holds in the case that $\phi$ has coefficients in the ring of integers $\CO_K$, as mentioned in Remark~\ref{algebraicintegers}.

We now prove Theorem~\ref{rigidity} using a result of Kawaguchi and Silverman on polynomials with equal canonical height functions~\cite{Kawaguchi}.
\begin{proof}[Proof of Theorem \ref{rigidity}]
By Theorem~\ref{main}, 
\begin{equation*} 
 \langle x^2, \phi \rangle = h_{\phi}(0) - \sum_{\nu \in M_K} r_{\nu} \int_{\vert \alpha  \vert_{\nu} < 1} \log \vert \alpha \vert_{\nu} d \mu_{\phi,\nu}.
 \end{equation*}
We have $0$ preperiodic for $\phi$ and so $h_\phi(0)=0$. The hypotheses imply that the support of $\mu_{\phi,\nu}$ is disjoint from the Berkovich open unit disk for each $\nu\in M_K$ (for $\nu$ non-archimedean, this is because the coefficients are integral). Therefore $\langle x^2, \phi \rangle=0$. By~\cite{PST}*{Theorem 3}, the canonical heights $h_{x^2}$ and $h_{\phi}$ are equal. 

By~\cite{Kawaguchi}*{Corollary 25}, if $\psi,\phi$ are polynomials of degree $\geq 2$ with $h_\psi=h_\phi$, then there exist a linear polynomial $f$ and a root of unity $\eta$ such that $\psi^f(x)=x^n$ and $\phi^f(x)=\eta x^d$ (where $\psi^f$ means $f^{-1}\circ \psi\circ f$). In our case where $\psi(x)=x^2$, the only possibility is that  $f(x)=x$ and $n=2$, which imply $\eta=1$. Thus $\phi(x)=x^d$.
\end{proof}
\begin{remark}
It might seem that we could strengthen Theorem~\ref{rigidity} by weakening the hypothesis to allow $\phi$ to either have good reduction at $\nu$ or satisfy the hypothesis of Lemma~\ref{rat}. However, the argument of Lemma~\ref{rat} shows that if $\nu(a_0)<0$ then $0$ is in the basin of attraction of $\infty$, and in particular is not preperiodic. So if Lemma~\ref{rat} is satisfied and $0$ is preperiodic, then $\phi$ has good reduction.
\end{remark}

\section{Explicit computations}

Because of the explicit nature of Theorem~\ref{main}, we are able to compute the Arakelov-Zhang pairing in several situations, extending the computations in \cite{PST}. 

First, we compute the Arakelov-Zhang pairing between $x^2$ and the Chebyshev polynomials $T_n(x)$, an important family of polynomials in dynamics. The polynomial $T_n$ is characterized by the equation $T_n(x + x^{-1}) = x^n + x^{-n}$ (see, e.g., \cite{Silverman} for basic facts about Chebyshev polynomials).

\begin{prop}\label{cheb}
Let $T_n(x)$ be a Chebyshev polynomial for $n \ge 2$. Then  $$\langle x^2, T_n \rangle = -\frac{1}{2 \pi} \int_{-1}^{1} \frac{ \log \vert x \vert}{\sqrt{1 - x^2/4}} dx = \frac{3 \sqrt{3}}{4 \pi} L(2, \chi) \approx 0.3231,$$ where $L(s, \chi)$ is the Dirichlet $L$-function associated to the nontrivial character with modulus $3$. 
\end{prop}

\begin{proof}
The Chebyshev polynomials are monic polynomials with integer coefficients, so Corollary \ref{integercoeff} applies. We note that the complex Julia set of the $T_n(x)$ for $n \ge 2$ is the interval $[-2, 2]$, and $0$ is preperiodic. It is known (see, e.g., \cite{comp}*{Example 2.6}) and can be easily verified that the canonical measure on the complex Julia set of $T_n$ is given by $$d\mu_{T_n} = \frac{1}{2 \pi} \frac{1}{\sqrt{1 - x^2/4}} dx,$$ where $dx$ is the standard Lebesgue measure. Then Corollary \ref{integercoeff} gives
$$\langle x^2, T_n \rangle = 0 - \int_{\vert x \vert \le 1} \log \vert x \vert d\mu_{T_n} =  -\frac{1}{2 \pi} \int_{-1}^{1} \frac{ \log \vert x \vert}{\sqrt{1 - x^2/4}} dx.$$
We claim that $$ -\frac{1}{2 \pi} \int_{-1}^{1} \frac{ \log \vert x \vert}{\sqrt{1 - x^2/4}} dx = \frac{1}{\pi} \int_{-\pi/3}^{\pi/3} \log \vert 2 + 2 \sin t \vert dt = \frac{3 \sqrt{3}}{4 \pi} L(2, \chi).$$
The later equality is shown by Smyth in an appendix to \cite{Boyd}. 
The substitution $x = 2 \sin(t/2)$ gives $$-\frac{1}{2 \pi} \int_{-1}^{1} \frac{\log \vert x \vert}{\sqrt{1 - x^2/4}} dx = -\frac{1}{2\pi} \int_{-\pi/3}^{\pi/3} \log \vert 2 \sin(t/2) \vert dt = -\frac{1}{2\pi} \int_{0}^{\pi/3} \log( 4 \sin^2(t/2)) dt.$$
Using the Pythagorean identity, this is equal to
\begin{equation*} \begin{split} -\frac{1}{2\pi} \int_{0}^{\pi/3} \log( 4 \sin^2(t/2)) & =-\frac{1}{2 \pi} \int_{0}^{\pi/3} (\log (2 - 2 \cos(t/2))+  \log(2 + 2 \cos(t/2))) dt\\
 &= -\frac{1}{\pi} \int_{0}^{\pi/6} (\log(2 - 2 \cos(t)) + \log(2 + 2 \cos(t)))dt \\
& = -\frac{1}{\pi} \int_{\pi/3}^{\pi/2} (\log(2 - 2 \sin(t)) + \log(2 + 2 \sin(t))) dt. \end{split} \end{equation*}
Substituting $t \to -t$ on the interval from $-\pi/3$ to $0$, we see that $$\frac{1}{\pi} \int_{- \pi/3}^{\pi/3} \log \vert 2 + 2 \sin t \vert dt = \int_{0}^{\pi/3} (\log(2 + 2 \sin(t)) + \log(2 - 2 \sin(t)))dt.$$
Then the result follows as $$\int_{0}^{\pi/2}  (\log(2 + 2 \sin(t)) + \log(2 - 2 \sin(t)))dt = \int_{0}^{\pi/2} \log(2 \cos(t)) dt = 0,$$
where the last equality is well-known.
\end{proof}

We are also able to extend the some of the explicit computations done in \cite{PST}. For example, Petsche-Szpiro-Tucker prove Proposition~\ref{quadPST}.
\begin{prop}\cite{PST}*{Proposition 19}\label{quadPST}
Let $\phi(x) = x^2 + c$ for $c \in K$. Then 
$$(1/2) h(c) - \log 3 \le \langle x^2, \phi \rangle \le (1/2) h(c) + \log 2.$$
\end{prop}

We show the following.
\begin{prop}\label{quad}
Let $\phi(x) = x^2 + c$ for $c \in K$. Then $$\langle x^2, \phi \rangle \ge h_{\phi}(0).$$ 
Suppose that for all archimedean place $\nu$ of $K$ we have that $\vert c \vert_{\nu} \ge 2 + \sqrt{2}$. 
Then 
$$\langle x^2, \phi \rangle = h_{\phi}(0).$$
\end{prop}

\begin{proof}
The first statement follows immediately from Corollary \ref{disjointcor}. Note that polynomials of the form $x^2 + c$ satisfy the hypothesis of Lemma~\ref{rat} for $\nu$ non-archimedean. In general, we show that if $\vert c \vert_{\nu} \ge 2 + \sqrt{2}$, then the Julia set of $\phi$ at an archimedean place $\nu$ is disjoint from the open unit disk. Indeed, if $\vert c \vert_{\nu} \ge 2 + \sqrt{2}$, then every point in the open unit disk lies in the basin of attraction of infinity at $\nu$. If $\vert x \vert_{\nu} > \frac{1 + \sqrt{1 + 4 \vert c \vert_{\nu}}}{2}$, then $\vert \phi(x) \vert_{\nu} > \vert x \vert_{\nu}$ and hence $\vert \phi^n(x) \vert_{\nu}$ grows geometrically. If $\vert c \vert_{\nu} \ge 2 + \sqrt{2}$ and $\vert a \vert_{\nu} < 1$, then 
$$\vert a^2 + c \vert_{\nu} > \vert c \vert_{\nu} - 1 \ge \frac{1 +  \sqrt{1 + 4 \vert c \vert_{\nu}}}{2}.$$
Thus if $\vert c \vert_{\nu} \ge 2 + \sqrt{2}$ for all archimedean places $\nu$, then the Julia set of $\phi$ at every place of $K$ is disjoint from the open unit disk, so the proposition follows from Corollary \ref{disjointcor}.
\end{proof}

We can combine Propositions \ref{quadPST} and \ref{quad}  to derive a bound on the difference between $h_{\phi}(0)$ and $(1/2)h(c)$. Since $$h_{\phi}(0) = \lim_{n \to \infty} \frac{h(\phi^n(0))}{2^n},$$ one may expect that $(1/2) h(\phi(0))$ is a reasonable approximation for $h_{\phi}(0)$. We show that this is a good approximation uniformly in $\phi$.

\begin{prop}
Let $\phi(x) = x^2 + c$. Then we have that $h_{\phi}(0) \le (1/2) h(\phi(0)) + \log 2$. 
\end{prop}

\begin{proof}
By Proposition \ref{quad}, $$\langle x^2, \phi \rangle \ge h_{\phi}(0).$$ By Proposition \ref{quadPST}, this implies that $h_{\phi}(0) \le (1/2)h(\phi(0)) + \log 2$. 
\end{proof}

Petsche-Szpiro-Tucker also consider the pairing between the squaring map $x^2$ and the map $\alpha - (\alpha - x)^2$, which is a conjugate of the squaring map by translation by $\alpha$. Let $\sigma(x)=x^2$, so that, for a M\"{o}bius transformation $f$, we have $\sigma^{f}(x) = f^{-1}(f(x)^2)$. For $a > 0, b \ge 0$, let $I(a, b) = - \int_{0}^{1} \log \min \{a, \vert b + e^{2 \pi i \theta} \vert\} d \theta.$  

\begin{prop}{\cite{PST}*{Proposition 18}}
 Suppose $f(x) = \alpha - x$ is defined over a number field $K$. Then 
$$\langle x^2, \sigma^{f} \rangle = h(\alpha) + \sum_{\nu \mid \infty} r_{\nu} I(1, \vert \alpha \vert_{\nu}).$$
\end{prop}

We simplify the proof of this this and extend it to the case of any linear $f$, which allows us to bound the difference between the standard height and canonical height in these cases.
\begin{prop}\label{conjugation}
Let $f(x)=ax+b\in K[x]$ with $a\neq 0$. Then 
$$\langle x^2, \sigma^{f} \rangle = h(b) + \sum_{\nu \mid \infty} r_{\nu} \log  \vert a \vert_{\nu} + \sum_{\nu \mid \infty} r_{\nu} I(\vert a \vert_{\nu}, \vert b \vert_{\nu}).$$
\end{prop}
In \cite{PST}*{Lemma 17}, Petsche-Szpiro-Tucker analyze $I(1, b)$ and show that it attains a global maximum at $1$. They compute that $$I(1,1) = \frac{3 \sqrt{3}}{4 \pi} L(2, \chi),$$ where as in Proposition \ref{cheb}, $L(s, \chi)$ is the Dirichlet $L$-function associated to the nontrivial character with modulus $3$. 
Using similar techniques and some casework, one can show that for fixed $a$, $I(a, b)$ attains a global maximum at $b = 1$, and that $I(a, 1) = 0$ for $a \ge 2$. For fixed $b$, $I(a, b)$ is a non-increasing function.
Note that $\sum_{\nu \mid \infty} r_{\nu} = 1$. Thus if $\vert a \vert_{\nu} \ge 1$ for all archimedean places $\nu$, $$\langle x^2, \sigma^{f} \rangle \le h(b) + \sum_{\nu \mid \infty} r_{\nu} \log  \vert a \vert_{\nu} + \frac{3 \sqrt{3}}{4 \pi} L(2, \chi).$$

\begin{proof}[Proof of Proposition~\ref{conjugation}]
Recall that the canonical measure of $\sigma(x) = x^2$ at an archimedean place $\nu$ (viewed as an embedding of $K$ into $\BC$) is the uniform measure on the unit circle. Since the canonical measure of $\sigma^{f}$ is the pushforward of the canonical measure of $\sigma$ under $f$, the canonical measure of $\sigma^{f}$ is the uniform measure on the circle of radius $a^{-1}$ centered at $f^{-1}(0)=-b/a$. Therefore
\begin{equation*} \begin{split}-\int_{\vert \alpha \vert_{\nu} < 1} \log \vert \alpha \vert_{\nu} d \mu_{\sigma^f, \nu} &= -\int_0^1 \log \min\{1, \vert \vert -b/a \vert_{\nu} + \vert a^{-1}\vert_{\nu} e^{2 \pi i \theta} \vert\} d \theta \\
&= -\int_{0}^{1} \log \left( \vert a^{-1} \vert_{\nu} \cdot \min\{\vert a \vert_{\nu}, \vert \vert b \vert_{\nu} +  e^{2 \pi i \theta} \vert\} \right) d \theta\\
&=  \log \vert a \vert_{\nu} + I(\vert a \vert_{\nu}, \vert b \vert_{\nu}),
\end{split}\end{equation*}
since $e^{2 \pi i \theta} + t$ parameterizes the unit circle centered at $t$ at constant speed at $\theta$ goes from $0$ to $1$. 
Note that $h_{\sigma^{f}}(x) = h(f(x))$, so $h_{\sigma^{f}}(0) = h(f(0))=h(b)$. Therefore Theorem~\ref{main} implies the result. 
\end{proof}

We now apply these explicit computations and Theorem~\ref{diffbound} to bound the height difference between the canonical height and the standard height. 

\begin{prop}
Let $T_n(x)$ denote the $n$th Chebyshev polynomial. Let $f(x) = ax + b$ be a M\"{o}bius transform defined over a number field $K$ with $\vert a \vert_{\nu} \ge 1$ for all archimedean places $\nu$. Let $\sigma^{f}(x) = f^{-1}(f(x)^2)$. Let $c = \frac{3 \sqrt{3}}{4 \pi} L(2, \chi) + \log 2$. Then for any $n\geq 2$ and any $x \in \BP^1(\overline{K})$,
$$h_{T_n}(x) - h(x) \le c.$$
$$h_{\sigma^{f}}(x) - h(x) \le c + h(b) + \sum_{\nu \mid \infty} r_{\nu} \log  \vert a \vert_{\nu}.$$
\end{prop}


\end{document}